\documentclass[14pt, a4paper]{amsart}

\usepackage[14pt]{extsizes}
\usepackage{amsmath}
\usepackage{amssymb}

\usepackage{enumerate}

\newtheorem{Theorem}{Theorem}
\newtheorem{Corollary}{Corollary}

\newtheorem{Remark}{Remark}

\begin{document}

\author{D.\,V.~Alekseevsky}
\address{Higher School of Modern Mathematics, MIPT, Moscow, Klimentovsky per. 1}
\email{alekseevskii.dv@mipt.ru}

\author{A.\,N.~Lavrov}
\address{Higher School of Modern Mathematics, MIPT, Moscow, Klimentovsky per. 1}
\email{lavrov.an@mipt.ru}

\title{Einstein manifolds and extended Poincar\'{e} algebras}

\maketitle

    \begin{abstract}

        For any pseudo-Euclidean space $V \simeq \mathbb{R}^{p, q}, p \geq 3$, module $W$ over the even Clifford algebra
        $\text{Cl}^{0}(V)$ and $V$-valued skew-symmetric bilinear $\mathfrak{so}(V)$-equivariant form $\Pi$ on $W$,
        V. C\'{o}rtes constructed noncompact quaternion-pseudo-K\"{a}hler homogeneous space. In particular, such homogeneous
        space is always Einstein and for $p=3$ it is Riemannian. The current paper is devoted to the generalization of this
        construction for $p = 3$ to the case $p \equiv 3 \pmod{4}$ and an irreducible $\text{Cl}^{0}(V)$-module. As a result
        we construct the new series of Riemannian Einstein homogoneous spaces $M_{p, q}$ with negative scalar curvature.
  
        

    \end{abstract}

    \footnotetext{This work is supported by the Russian Science Foundation \textnumero 25-21-00918, 
                       https://rscf.ru/project/25-21-00918/}

    \section{Introduction}

    The problem of description of Einstein metrics on manifolds is classical in differential geometry.
    Riemannian manifold is called \textit{Einstein} if the corresponding Ricci curvature tensor is a scalar multiple of
    the metric tensor. Einstein equation $\text{Ric}_{g} = \lambda \cdot g$, $\lambda \in \mathbb{R}$, is a system of
    non-linear partial differential equations of the second order. The structure of the solution space of Einstein equation
    is still not well understood (see \cite{Bes, Ber}). It is natural to require some symmetries for solutions, i. e. to consider
    metrics invariant under action of some Lie group. In the homogeneous case Einstein equation turns into the complicated
    system of algebraic equations. Following the excellent review \cite{J} we shortly describe the main results in this direction.

    Riemannian homogeneous spaces $M$ can be divided into three groups depending in the sign of the scalar curvature $\text{sc}$.
    For $\text{sc} > 0$, according to Myers's theorem, $M$ have to be compact. For such spaces the following results are
    known. Firstly, the moduli space of invariant Einstein metrics with unit volume on compact homogeneous space
    has the finite number of connected compact components (see \cite{BWZ}). Moreover, for the cases when the moduli space
    is possible to completely describe it is finite. However, the question ``is this moduli space always finite?'' is still open.
    Next, the full classification of Einstein homogeneous spaces with the positive scalar curvature is obtained only for
    dimension less or equal to 7 (see \cite{ADF, N1, N2}). Concerning the existence questions, for example, it is known that if
    $G$ is a semisimple group and $H$ is its closed subgroup then the simply connected homogeneous space $M = G/ H$
    of the dimension less or equal to $\leq 11$ admits $G$-invariant Einstein metric (see \cite{BK}). There are also sufficient
    conditions for existence of Einstein metric in terms of contraction of some simplicial complex associated with $M$ (see \cite{B, Gr}).
  
    For $\text{sc} = 0$ Riemannian homogeneous space $M$ can be either compact or non-compact, but according to
    Alekseevsky-Kimelfeld theorem (see \cite{AK}) it is necessary flat. In other words, it is always isometric to
    the Cartesian product of $n$-dimensional flat torus $\mathbb{T}^{n}$ and Euclidean space $\mathbb{R}^{m}$,
    i. e. $M \simeq \mathbb{T}^{n} \times \mathbb{R}^{m}$.

    For $\text{sc} < 0$, according to Bochner's theorem, Riemannian homogeneous spaces are non-compact.
    For Riemannian Einstein homogeneous spaces of such type the series of structural results was obtained.
    First of all, any noncompact Riemannian Einstein homogeneous space is a \textit{solvmanifold}, i. e. it is a simply
    connected solvable Lie group $S$ (which is diffeomorphic to $\mathbb{R}^{n}$) equipped with left-invariant metric.
    This statement was firstly stated as a hypothesis in 1975 year in \cite{A} and it was recently proved in \cite{BL}.
    The geometry of Riemannian solvmanifold is completely determined by the \textit{metric Lie algebra} $(\mathfrak{s}, g)$.
    Earlier it was shown in \cite{L1} that any Einstein solvmanifold is \textit{standard}, i. e. the orthogonal compliment to
    the derived Lie subalgebra $\mathfrak{a} := [\mathfrak{s}, \mathfrak{s}]^{\perp}$ is an abelian subalgebra.
    The dimension of the abelian part of $\mathfrak{a}$ is called a \textit{rank} of the standard solvmanifold,
    i. e. $\text{rk } S := \dim \mathfrak{a}$. The study of standard Einstein solvmanifolds was initiated in \cite{Heb}.
    In this paper it was that there is no more than one Einstein left-invariant metric up to isometry and homothety
    on standard solvmanifold.

    A nilpotent Lie algebra is called \textit{Einstein nilradical} if it can be represented as nilradical of a solvable Lie algebra
    which corresponds to some Einstein solvmanifold. The classification of Einstein sovmanifolds can be reduced to
    the classification of Einstein nilradicals (see \cite{W2}). It is proved in \cite{L2} that nilpotent Lie algebra $\mathfrak{n}$ is
    Einstein nilradical if and only if it admits a \textit{nilsoliton metric}. In particular, the restriction of Einstein metric of
    solvmanifold onto corresponding nilradical induces a nilsoliton metric. It is known that there is no more than one
    nilsoliton metric up to isometry and homothety on nilpotent Lie algebra.
    
    The next nilpotent Lie algebras $\mathfrak{n}$ are Einstein nilradicals:
    abelian, $\mathfrak{n}$ has a codimension one abelian ideal (see \cite{L3}),
    $\dim\mathfrak{n} \leq 6$ (see \cite{L3, W1}), the Lie algebra of strictly upper triangular matrices (see \cite{P}).
    Free and filiform Einstein nilradicals are classified in \cite{Nik2, Nik3}. The necessary and sufficient conditions for
    nilpotent Lie algebra with a nice basis (see \cite{LW, CR, G}) to be Einstein nilradical are given in \cite{Nik1}.
    
    In the present paper the infinite series of homogeneous Einstein spaces of negative scalar curvature is constructed.
    The approach is based on the construction of noncompact quaternion-K\"{a}hler homogeneous spaces considered by V. C\'{o}rtes
    in \cite{C}. For pseudo-Euclidean space $V = \mathbb{R}^{p, q}$ consider a graded extension of the Poincar\'{e} algebra
    $\text{Lie}\big( \text{Isom}(V) \big)$ by a module $W$ over the even subalgebra $\text{Cl}^{0}(V)$ of
    the Clifford algebra $\text{Cl}(V)$ and $V$-valued skew-symmetric bilinear $\mathfrak{so}(V)$-equivariant form
    $\Pi: \wedge^{2} W \longrightarrow V$. All such extensions called \textit{extended Poincar\'{e} algebras} are classified in \cite{AC}.
    Next, consider the further extension of the obtained Lie algebra by a derivation $D$ with eigenvalues $(0, \frac{1}{2}, 1)$.
    Denote the Lie group corresponding to the constructed Lie algebra by $G$. All noncompact quaternion-K\"{a}hler homogeneous
    spaces can be obtained as quotient manifolds of the form $G/K$ where $K$ is a maximal compact subgroup and $p = 3$.
    We prove (see Theorem \ref{MainTheorem}) that for $p \equiv 3 \pmod 4$\footnote{
    The case $p > 3$ is also considered in \cite{C}, but from the different point of view. Namely, the quotient manifold $G/H$,
    where the subgroup $H$ is not compact, is considered. In this way we obtain noncompact quaternion-pseudo-K\"{a}hler homogeneous
    spaces (i. e. equipped with pseudo-Riemannian metric).} and an irreducible $\text{Cl}^{0}(V)$-module $W$ this
    construction gives (Riemannian) Einstein homogeneous spaces of negative scalar curvature which are no longer
    quaternion-K\"{a}hler. Moreover, the corresponding homogeneous spaces are parametrized by the signature $(p, q)$
    up to isometry (see Corollary \ref{HomogeneousSpacesIsomorphism}).
    
    Note that since $K$ is a maximal compact subgroup the constructed homogeneous spaces can be always represented as
    solvable Lie groups equipped with left-invariant Einstein metrics. Moreover, one can show that nilradicals of
    the coresponding solvable Lie algebras have a nice basis. Therefore, the linear conditions of \cite[Th. 3]{Nik1} are applicable in
    our case. Nevertheless, due to the dimension of the series of considered nilradicals is not bounded and due to
    complicated structure of these nilradicals the problem of solving the corresponding system of linear equations
    seems quite difficult. Our proof is based on the straithforward computation of the Ricci tensor and
    checking Eqinstein equation.

    Also it is worth to note that the considered construction of homogenoeus spaces is similar to the construction of so-called
    \textit{generalized Heisenberg groups} introduced by A. Kaplan in \cite{K} and their extensions called
    \textit{Damek-Ricci spaces} (see \cite{DR, M, BTV}). One of the differences is that in the construction of generalised Heisenberg
    group and Damek-Ricci space only positive definite signature $(p, q) = (p, 0)$ (for any $p > 0$) is considered.
    On the contrary, we consider arbitrary\footnote{
    One can show that for the case $q = 0$ our construction will give exactly the Damek-Ricci spaces.}  $q \geq 0$. 

    \section{Riemannian homogeneous space $M_{p, q}$}
  
    Let $V = \mathbb{R}^{p, q}$ be the standard pseudo-Euclidean space with indefinite scalar product $\langle \cdot, \cdot \rangle$
    of the signature $(p, q)$. We will denote the dimension of $V$ by $n := \dim V = p + q$. Remind that the Clifford algebra
    $\text{Cl}(V)$ is the quotient of the tensor algebra $T(V)$ by the two-sided ideal\footnote{Note that the other
    definition of Clifford algebra is often used in the literature in which the considered ideal is generated by elements 
    $v \otimes v - \langle v, v \rangle \cdot 1$, $v \in V$.} generated by elements of the form
    $v \otimes v + \langle v, v \rangle \cdot 1$, $v \in V$. The reflection $v \mapsto (-v)$ on $V$ induces
    $\mathbb{Z}_{2}$-grading $\text{Cl}(V) = \text{Cl}^{0}(V) + \text{Cl}^{1}(V)$ where $\text{Cl}^{0}(V)$ is called
    the even Clifford algebra. $\text{Pin}(V) \subset \text{Cl}(V)$ is the group generated by all vectors $v \in V$ satisfying
    $\langle v, v \rangle = \pm 1$. Its subgroup $\text{Spin}(V) := \text{Pin}(V) \cap \text{Cl}^{0}(V)$ consisting of
    all even elements is called the spin group. 

    The Lie algebra $\mathfrak{spin}(V) := \text{Lie}\big( \text{Spin}(V) \big) \simeq 
    \text{span}\{ [u, v] \ | \ u, v \in V \} \subset \text{Cl}^{0}(V)$ 
    where $[u, v] := uv - vu$ is commutator in the Clifford algebra, is isomorphic to the Lie algebra
    $\mathfrak{so}(V) := \text{Lie}\big( SO(V) \big)$ which can be identified with the bivector space $\wedge^{2} V$ by the formula 
    \begin{equation}
        u\wedge v : w \mapsto  \langle v, w \rangle u -   \langle u, w  \rangle v. 
    \end{equation}
    Isomorphism $\mathfrak{spin}(V) \simeq \wedge^{2} V$ is given by the map $[u, v] \mapsto (-4) \cdot (u \wedge v)$
    for any $u, v \in V$. In particula, $u \cdot v \mapsto (-2) \cdot (u \wedge v)$ for $u \perp v \in V$.

    Consider an irreducible $\text{Cl}^{0}(V)$-module $W$ of the dimension $d := \dim W$. The action of $\text{Cl}^{0}(V)$
    on $W$ induces the irreducible linear representation of $\mathfrak{so}(V)$ on $W$ 
    ($\text{ad}_{u \wedge v} := -\frac{1}{2} uv$ for any orthogonal vectors $u, v \in V$).
    According to \cite[Th. 5]{C}, for $V = \mathbb{R}^{p,q}$ with $p \equiv 3 \pmod{4}$ the space of
    $\mathfrak{so}(V)$-equivariant  $V$-valued 2-forms on $W$  is one-dimensional,
    \begin{equation}
        \text{dim}(\wedge^{2}W^{*} \otimes V)^{\mathfrak{so}(V)} = 1. 
    \end{equation}
    Therefore, there exists non-trivial $\mathfrak{so}(V)$-equivariant linear map $\Pi : \bigwedge^2 W \longrightarrow V$
    which defines $\mathbb{Z}_{2}$-graded Lie algebra 
    \begin{equation}
        \mathfrak{p}(W, \Pi) = \mathfrak{p}_{0} + \mathfrak{p}_{1} = \big( \mathfrak{so}(V) + V \big) + W.
    \end{equation} 
    This Lie algebra is called extented Poincar\'{e} algebra and it was introduced in \cite{AC}\footnote{In \cite{AC}
    a $\text{Cl}^{0}(V)$-module is not necessary irreducible.}.
    
    Lie algebra $\mathfrak{p}(W, \Pi)$ admits the following derivation $D$:
    \begin{equation}
        \text{ad}_{D}|_{V} = \text{Id}, \ \ \ \text{ad}_{D}|_{W} = \frac{1}{2}\text{Id}, \ \ \ \text{ad}_{D}|_{\mathfrak{so}(V)} = 0.
    \end{equation} 
    This derivation defines $\mathbb{Z}_{2}$-graded Lie algebra
    \begin{equation}
        \mathfrak{g}(W, \Pi) = \mathfrak{g}_{0} + \mathfrak{g}_{1} 
            = \big( \mathbb{R}D + \mathfrak{p}_{0} \big) + \mathfrak{p}_{1}
            = \big( \mathbb{R}D + \mathfrak{so}(V) + V \big) + W.
    \end{equation}
    Denote the corresponding simply connected Lie group by $G(W, \Pi)$.  Note that the identity component of
    the spin group $\text{Spin}_{0}(V) \subset \text{Spin}(V)$ can be considered as a closed Lie subgroup of $G(W, \Pi)$
    corresponding to the Lie subalgebra $\mathfrak{so}(V) \subset \mathfrak{g}(W, \Pi)$ (see \cite[Cor. 3]{C}).

    Next, denote the standard orthonormal basis of the pseudo-Euclidean space $V = \mathbb{R}^{p, q}$ by 
    $(e_{1}, .., e_{p},$ $e'_{1}, .., e'_{q})$ such that $(e_{1}, .., e_{p})$ is a basis of the subspace $E := \mathbb{R}^{p,0} \subset V$,
    while $(e'_{1}, .., e'_{q})$ is a basis of the subspace $E' := \mathbb{R}^{0, q} \subset V$. Therefore, we have the orthogonal
    decomposition $V = E \oplus E'$ which induces the decomposition  $\mathfrak{so}(V) \simeq E \wedge E + E' \wedge E' + E \wedge E'$.
    Also note that the basis determines orientation of  $V$, $E$ and $E'$. 
  
    Let $K$ be a maximal connected subgroup of $\text{Spin}_{0}(V)$ which preserves the othogonal decomposition $V = E \oplus E'$.
    We have that $K = \text{Spin}(E) \times \text{Spin}(E') \simeq \text{Spin}(p) \times \text{Spin}(q)$, in particular,  it is
    a maximal compact subgroup of $\text{Spin}_{0}(V)$ which corresponds to the Lie subalgebra
    $\mathfrak{k} := \mathfrak{so}(E) + \mathfrak{so}(E') \simeq E \wedge E + E' \wedge E' \subset \mathfrak{so}(V)$.
    As well as $\text{Spin}_{0}(V)$ the group  $K$ can be considered as a closed Lie subgroup of $G(W, \Pi)$.
    Now consider the reductive homogeneous space $M(W, \Pi) := G(W, \Pi)/K$ with the following reductive decomposition of
    the Lie algebra
    \begin{equation}\label{ReductiveDecomposition}
        \mathfrak{g}(W, \Pi) = \mathfrak{k} + \mathfrak{m}, \ \ \text{where} \ \mathfrak{k} = E \wedge E + E' \wedge E', 
    \end{equation}
    \begin{equation*}
        \mathfrak{m} = \mathbb{R}D + E \wedge E' + E + E' + W.
    \end{equation*}
    Note that $\mathbb{R}D$, $E \wedge E'$, $ E$, $E'$ and $W$ are irreducible $\text{Ad}_{K}$-modules which are
    pairwise non-equivalent for $q > 1$. For the case $q = 1$ we have that $E' \simeq \mathbb{R}D$ and $E \wedge E' \simeq E$
    as $\text{Ad}_{K}$-modules. Next, the tangent space $M(W, \Pi)$ at the point $o := [eK] \in M(W, \Pi)$ is isomorphic to
    the subspace $\mathfrak{m} \subset \mathfrak{g}(W, \Pi)$, i. e. $T_{o} M(W, \Pi) \simeq \mathfrak{m}$. The left-invariant
    metrics on $M$ corresponds to $\text{Ad}_{K}$-invariant scalar products in $\mathfrak{m}$. Note that
    the Ricci tensor of any left-invariant metric at the point $o \in M(W, \Pi)$ can be computed in terms of the scalar product
    in $\mathfrak{m}$ and Lie bracket. 

    Following \cite{C}, we define bilinear form $b$ on $W$ in the formula
    \begin{equation}\label{canonical_form}
        b(s, s') := \langle e_{1}, [e_{2}..e_{p} s, s'] \rangle \ \text{ for any } s, s' \in W.
    \end{equation}
    Also consider the element of the Clifford algebra $\omega := e_{1} ... e_{p} \in \text{Cl}(E) \subset \text{Cl}(V)$.
    Note that due to $e_{i}^{2} = -1$, $1 \leq i \leq p$, we have the equality $\omega^{2} = (-1)^{\frac{p(p+1)}{2}}$.
    In particular, for $p \equiv 3 \pmod{4}$ we have $\omega^{2} = 1$.
    \begin{Theorem}\label{b_properties}
        For $p \equiv 3 \pmod{4}$ the bilinear form $b$ satisfies the properties:
        \begin{enumerate}[(i)]
            \item $b$ does not depend on the choice of positively-oriented basis of $E$,
            \item $b$ is symmetric, definite, nondegenerate form, 
            \item $E \wedge E'$ acts on $W$ by $b$-symmetric endomorphisms, 
                while $\wedge^{2}E$ and $\wedge^{2}E'$ by $b$-skew-symmetric endomorphisms,
            \item for any $v \in V$ and $s, s' \in W$ the following equality holds: \\
                $\langle [s, s'], v \rangle = b(\omega v s, s')$.
        \end{enumerate}
    \end{Theorem}
    \noindent\textit{Proof:} 
    See proof of (i)-(iii) in \cite[Thm. 1]{C} and \cite[Lemma 3]{C}. To prove (iv) it is enough to check the equality for
    all basis vectors of $V$. Note that for $v=e_{1}$ the equality follows from the definition (\ref{canonical_form}) of $b$.
    For $v=e_{i}, i = 2,.., p,$ we have
    \begin{equation*}
        \langle [s, s'], e_{i} \rangle = \langle [s, s'], \text{ad}_{e_{i} \wedge e_{1}} e_{1} \rangle
            = \langle \text{ad}_{e_{1} \wedge e_{i}} [s, s'], e_{1} \rangle
    \end{equation*}
    \begin{equation*}
            = \frac{1}{2} \langle [e_{i}e_{1} s, s'], e_{1} \rangle + \frac{1}{2} \langle [s, e_{i}e_{1} s'], e_{1} \rangle
            = \frac{1}{2} b(\omega e_{i} s, s') + \frac{1}{2} b(e_{2}..e_{p} s, e_{1}e_{i} s')
    \end{equation*}
    \begin{equation*}
            = \frac{1}{2} b(\omega e_{i} s, s') - \frac{1}{2} b(e_{1}e_{i} e_{2}..e_{p} s,  s') = b(\omega e_{i} s, s').
    \end{equation*}
    The computation for $v=e'_{j}, j =1,.., q,$ is similar
    \begin{equation*}
        \langle [s, s'], e'_{j} \rangle = \langle [s, s'], \text{ad}_{e'_{j} \wedge e_{1}}e_{1} \rangle
            = \langle \text{ad}_{e_{1} \wedge e'_{j}} [s, s'], e_{1} \rangle
    \end{equation*}
    \begin{equation*}
            = \frac{1}{2} \langle [e'_{j}e_{1} s, s'], e_{1} \rangle + \frac{1}{2} \langle [s, e'_{j}e_{1} s'], e_{1} \rangle
            = \frac{1}{2} b(\omega e'_{j} s, s') + \frac{1}{2} b(e_{2}..e_{p} s, e_{1}e'_{j} s')
    \end{equation*}
    \begin{equation*}
            = \frac{1}{2} b(\omega e_{i} s, s') + \frac{1}{2} b(e_{1}e'_{j} e_{2}..e_{p} s,  s') = b(\omega e'_{j} s, s').
    \end{equation*}
    \hfill$\Box$

    Now we consider the positive-definite scalar product $g_{V}$ in $V$ such that  
    $g_{V}|_{E} = \langle -, - \rangle|_{E}$, but $g_{V}|_{E'} = (-1)\cdot\langle -, - \rangle|_{E'}$.
    It induces the positive-definite scalar product $g_{E \wedge E'}$ in $E \wedge E'$. 
    Since the sign of the form $b$ can always be changed by replacing $\Pi$ to $(-1) \cdot \Pi$ we will assume that
    $b$ is negative-definite. We define the positive-definite scalar product in the whole $\mathfrak{m}$ in the following way
    \begin{equation}
        g = g_{\mathbb{R}D} \oplus g_{E \wedge E'} \oplus g_{V} \oplus (-b).
    \end{equation}
    where $g_{\mathbb{R}D}$ is defined by the equality $g_{\mathbb{R}D}(D, D) = 1$. Due to that $g$ is $\text{Ad}_{K}$-invariant
    it can be extended to the left-invariant Riemannian metric $g_{M(W, \Pi)}$ on the homogeneous space $M(W, \Pi) = G(W, \Pi)/K$.

    The following theorem is the partial case of \cite[Th. 6]{C} and \cite[Th. 7]{C}.
    \begin{Theorem}\label{PoincareAlgebrasIsomorphism}
        For fixed $p \equiv 3 \pmod{4}$ and $q \geq 0$ the isomorphism class of Lie algebra $\mathfrak{p}(W, \Pi)$
        does not depend on the choice of irreducible $\emph{Cl}^{0}(V)$-module $W$ and non-trivial
        $\mathfrak{so}(V)$-equivariant linear map $\Pi: \bigwedge^2 W \longrightarrow V$.
    \end{Theorem}
    \noindent\textit{Proof:} 
    From $\text{dim}(\wedge^{2}W^{*} \otimes V)^{\mathfrak{so}(V)} = 1$ it follows that arbitrary two
    $\mathfrak{so}(V)$-equivariant linear maps $\Pi, \Pi': \bigwedge^2 W \longrightarrow V$ are related by the equality
    $\Pi' = \lambda \Pi, \lambda \in \mathbb{R}^{*}$. If $\lambda = \mu^{2} > 0$ then the isomorphism is given by the map
    $\phi: \mathfrak{p}(W, \Pi) \longrightarrow \mathfrak{p}(W, \Pi')$ such that $\phi|_{\mathfrak{so}(V) + V} = \text{Id}$,
    $\phi|_{W} = \mu^{-1} \cdot \text{Id}$. For $\lambda = -1$ the isomorphism is given by the map
    $\alpha: \mathfrak{p}(W, \Pi) \longrightarrow \mathfrak{p}(W, \Pi')$ such that
    $\alpha|_{\mathfrak{so}(V) + W} = \text{Id}$, $\alpha|_{V} = (-1) \cdot \text{Id}$. It implies that
    $\mathfrak{p}(W, \Pi) \simeq \mathfrak{p}(W, \Pi')$ for any two non-trivial $\Pi, \Pi': \bigwedge^2 W \longrightarrow V$. 

    Next, note that for $p \equiv 3 \pmod{4}$, $q \not\equiv 3 \pmod{4}$ there exists\footnote{It follows from the isomorphism
    $\text{Cl}^{0}(\mathbb{R}^{p, q}) \simeq \text{Cl}(\mathbb{R}^{p-1, q})$ and the classification of Clifford algebras (see \cite[Table II]{LM})}
    only one irreducible $\text{Cl}^{0}(V)$-module $\Sigma$ up to isomorphism of $\text{Cl}^{0}(V)$-modules.
    On the other hand, for $p \equiv 3 \pmod{4}$, $q \equiv 3 \pmod{4}$ there exists exactly two irreducible non-equivalent
     $\text{Cl}^{0}(V)$-modules $\Sigma^{+}$ and $\Sigma^{-}$. The module $\Sigma^{-}$ can be obtained from $\Sigma^{+}$
     by the automorphism\footnote{Taking into account
    $\text{Cl}^{0}(\mathbb{R}^{p, q}) \simeq \text{Cl}(\mathbb{R}^{p-1, q})$ one can show that the automorphism $\hat{\psi}$
    corresponds to the involution on $\text{Cl}(\mathbb{R}^{p-1, q})$ induced by the reflection $v \mapsto (-1) \cdot v, v \in \mathbb{R}^{p-1, q}$.}
    $\hat{\psi} \in \text{Aut}( \text{Cl}^{0}(V) )$ induced by the orthogonal map $\psi \in  O(V)$ where $\psi(e_{1}) = e_{1}$,
    but $\psi(e_{i}) = -e_{i}$ for $2 \leq i \leq p$, and $\psi(e'_{j}) = -e'_{j}$ for $1 \leq j \leq q$.
    In its turn, this automorphism induces the automorphism $\Psi$ of Lie algebra $\mathfrak{p}(\Sigma^{+}, \Pi)$ such that
    $\Psi|_{V + \mathfrak{so}(V)} = \hat{\psi}|_{V + \mathfrak{so}(V)}$ and $\Psi|_{\Sigma^{+}} = \text{Id}$. So we obtain the isomorphism
    $\mathfrak{p}(\Sigma^{+}, \Pi) \simeq \Psi\big( \mathfrak{p}(\Sigma^{+}, \Pi) \big) = \mathfrak{p}(\Sigma^{-}, \Pi)$.
    \hfill$\Box$

    \noindent
    From Theorem \ref{PoincareAlgebrasIsomorphism} it follows immediatelly that the isomorphism class of Lie algebra
    $\mathfrak{g}(W, \Pi)$ and, correspondingly, the isomorphism class of Lie group $G(W, \Pi)$ is completely determined
    by the signature $(p, q)$ under the condition that $p \equiv 3 \pmod{4}$. The maps considered in Theorem \ref{PoincareAlgebrasIsomorphism}
    induce the isomorphisms of Li algebras of the form $\mathfrak{g}(W, \Pi)$ which respect the reductive decomposition
    (\ref{ReductiveDecomposition}). Moreover, the corresponding isomorphisms restricted onto $\mathfrak{m}$ are isometries.
    So we obtain the following statement.

    \begin{Corollary}\label{HomogeneousSpacesIsomorphism}
        For $p \equiv 3 \pmod{4}$ the Riemannian homogeneous space $M(W, \Pi)$ is determined by the signature $(p, q)$
        up to isometry.
    \end{Corollary}

    Denote\footnote{For $p = 3$ the introduced notation is consisted with the notations from \cite[Th. 9]{C} in the following way:
    $M_{3, q} \simeq M(q, 1)$ for $q \not\equiv 3 \pmod{4}$, and $M_{3, q} \simeq M(q, 1, 0) \simeq M(q, 0, 1)$
    for $q \equiv 3 \pmod{4}$.} the Riemannian homogeneous space $M(W, \Pi)$ by $M_{p, q}$ up to isometry $M_{p, q}$.
    The following statement holds.
    \begin{Theorem}\label{DimensionRank}
        Consider the function $N: \mathbb{Z}_{\geq 0} \longrightarrow \mathbb{Z}_{\geq 0}$ defined by the relations:
        $N(0)=N(1)=N(2)=N(3)=1$, $N(4) = 2$, $N(5) = 4$, $N(6)=N(7)=8$, and $N(k + 8) = 16 \cdot N(k)$ for all $k \geq 0$.
        The dimension and the rank of $M_{p, q}$ have the following form: $\dim M_{p, q} = (p+1)(q+1) + 4 \cdot N(p + q - 3)$,
        $\text{\emph{rk} } M_{p, q} = 1 + \min(p, q)$.
    \end{Theorem}
    \noindent\textit{Proof:} 
    The dimension $M_{p, q}$ can be computed as follows
    \begin{equation*}
        \dim M_{p, q} = \dim \mathfrak{m} = \dim \mathbb{R}D + \dim E + \dim E' + \dim E \wedge E' + \dim W
    \end{equation*}
    \begin{equation*}
            = 1 + p + q + pq + d = (p + 1)(q + 1) + d, \text{ where } d = \dim W.
    \end{equation*}
    Due to the isomorphisms of Clifford algebras $\text{Cl}^{0}(\mathbb{R}^{p, q}) \simeq \text{Cl}(\mathbb{R}^{p-1, q})$ and
    $\text{Cl}(\mathbb{R}^{r, s}) \simeq \text{Cl}(\mathbb{R}^{r+4, s-4})$ we obtain that for $p \equiv 3 \pmod{4}$
    we have the isomorphism $\text{Cl}^{0}(\mathbb{R}^{p, q}) \simeq \text{Cl}^{0}(\mathbb{R}^{3, p + q - 3})$.
    Consequently, the dimension of an irreducible $\text{Cl}^{0}(\mathbb{R}^{p,q})$-module coincides with the dimension
    of an irreducible $\text{Cl}^{0}(\mathbb{R}^{3, p+q-3})$-module. On the other hand, according to\cite[Pr. 10]{C}, the value
    of the function $N(k)$ is the quaternionic dimension\footnote{The quaternionic dimension of a quaternionic vector space $W$
    is $\text{dim}_{\mathbb{H}} W := \frac{1}{4} \dim W$.} of an irreducible $\text{Cl}^{0}(\mathbb{R}^{3, k})$-module
    for any $k \geq 0$. From this we obtain that $d = \dim W = 4 \cdot N(p+q-3)$.

    To compute the rank of solvmanifold $M_{p,q}$ we describe the maximal completely solvable subalgebra
    $\mathfrak{s} \subset \mathfrak{g}(\Pi)$ complementary to $\mathfrak{k}$ (see \cite{A1, A2}). To do this we
    introduce the notation $m := \text{min}(p, q)$. We have the orthogonal decomposition
    $V = \mathbb{R}^{p, q} = \mathbb{R}^{m, m} \oplus \mathbb{R}^{0, q-p}$ for $p \leq q$,
    and $V = \mathbb{R}^{p, q} = \mathbb{R}^{m, m} \oplus \mathbb{R}^{p-q, 0}$ for $p > q$.
    Consider the isotropic decomposition $\mathbb{R}^{m, m} = P + Q$ where $P$, $Q$ are maximal isotropic subspaces, $P^{*} \simeq Q$.
    For the case $p \leq $ we denote $\mathbb{R}^{0, q-p}$ by $R$, otherwise we denote $\mathbb{R}^{p-q, 0}$ by $R$.
    So we have the decomposition $V = P + Q + R$. Let $(p_{i})$, $i = 1,.., m$, be a basis of $P$ and $(q_{i})$, $i = 1,.., m,$
    be the dual basis of $Q$ . The subalgebra $\mathfrak{s}$ has the following form
    \begin{equation}
        \mathfrak{s} = \mathfrak{a} + \mathfrak{n} = \bigg( \mathbb{R}D + \text{span}\{ p_{i} \wedge q_{i} \ | \ i=1,..,m \} \bigg)
    \end{equation}
    \begin{equation*}
            + \bigg( \text{span}\{ p_{i} \wedge q_{j} \ | \ 1 \leq i < j \leq m \} + P \wedge P + P \wedge R + V + W \bigg).
    \end{equation*}
    Therefore, we have $\text{rk } M = \dim \mathfrak{a} = 1 + m = 1 + \text{min}(p,q)$.
    \hfill$\Box$

    \section{The metric $g_{M}$ is Einstein}

    \subsection{Computation of Nomizu operators}
    Ricci curvature can be expressed through so-called Nomizu operators. By definition, \textit{Nomizu operator} $L_{X}$
    associated with vector field $X$ on $M_{p, q}$ is the tensor field given by the formula $L_{X}Y := - \nabla_{Y}X$ where
    $Y$ is arbitrary vector field. In particular, $L_{X}|_{T_{o}M_{p,q}}$ is an endomorphism of the tangent space $T_{o}M_{p,q} \simeq \mathfrak{m}$. 
    On the other hand, any element of the Lie algbera $x \in\mathfrak{g}(W, \Pi)$ induces the Killing vector field $\alpha(x)$
    on $M$ which satisfies $\alpha(x)|_{o} = \pi(x) \in \mathfrak{m} \subset \mathfrak{g}(W, \Pi)$ where
    $\pi: \mathfrak{g}(W, \Pi) \longrightarrow \mathfrak{m}$ is the projection onto $\mathfrak{m}$ with respect to
    the reductive decomposition (\ref{ReductiveDecomposition}). For any element $x \in \mathfrak{g}(W, \Pi)$ we denote
    the endomorphism  $L_{\alpha(x)}|_{T_{o}M_{p,q}} \in \text{End}(\mathfrak{m})$ just by $L_{x}$. For simplicity we will also call
    such endomorphisms $L_{x}, x \in \mathfrak{g}(W, \Pi),$ Nomizu operators. The endomorphism $L_{x}, x \in \mathfrak{g}(W, \Pi),$
    satisfies the following equality for any $y, z \in \mathfrak{m}$ (see \cite{AVL})
    \begin{equation}\label{NomizuOperator}
        g(L_{x} y, z) = -\frac{1}{2} \Big( g\big(\pi([x, y]), z\big) - g\big(\pi([y, z]), x\big) + g\big(\pi([z, x]), y\big) \Big).
    \end{equation}
    From this one can see that Nomizu operators are skew-symmetric with respect to $g$. Also they satisfy the following equality
    \begin{equation}\label{NomizuIdentity}
        L_{y}\pi(x) - L_{x}\pi(y) = \pi([x, y]) \text{ for any } x, y \in \mathfrak{g}(\Pi).
    \end{equation}
    Along with Nomizu operators we will also use the operators $R_{x}: \mathfrak{g}(W, \Pi)\rightarrow\mathfrak{m}, x \in \mathfrak{m},$
    defined as follows
    \begin{equation}
        R_{x} y := L_{y} x = L_{x}\pi(y) + \pi([x, y]). 
    \end{equation}
    The Riemann curvature tensor of the metric $g_{M_{p,q}}$ at $o \in M_{p,q}$ can be expressed through  Nomizu operators
    like this (see \cite[Sec. 2.3]{C})
    \begin{equation}\label{RiemannTensorFormula}
        R(x, y) = [L_{x}, L_{y}] + L_{[x, y]} \text{ for any } x, y \in \mathfrak{m}.
    \end{equation}

    \begin{Theorem}\label{NomizuOperatorsTheorem}
        Nomizu operators for $g_{M}$ are the following:
        \begin{enumerate}[(i)]
            \item $L_{x} = 0$ for any $x \in \mathbb{R}D + E \wedge E'$, 
            \item $L_{x} = - \text{\emph{ad}}_{x}$ for any $x \in E \wedge E + E' \wedge E'$,
            \item $L_{e_{i}} = -D \otimes e_{i}^{*} + \sum\limits_{j=1}^{q} (e_{i} \wedge e'_{j}) \otimes {e'_{j}}^{*} 
                    - \text{\emph{ad}}_{e_{i}} - \frac{1}{2}\omega e_{i}$ where $1 \leq i \leq p$,
            \item $L_{e'_{j}} = -D \otimes {e'_{j}}^{*} + \sum\limits_{k=1}^{p} (e_{k} \wedge e'_{j}) \otimes e_{k}^{*} 
                    - \text{\emph{ad}}_{e'_{j}} + \frac{1}{2}\omega e'_{j}$ where $1 \leq j \leq q$,
            \item $L_{s} = -\frac{1}{2} D \otimes s^{*} 
                    + \frac{1}{2} \sum\limits_{i=1}^{p}\sum\limits_{j=1}^{q} (e_{i} \wedge e'_{j}) \otimes \big( (e_{i} e'_{j}) s \big)^{*}
                    - \frac{1}{2}\text{\emph{ad}}_{s}|_{W} - \text{\emph{ad}}_{s}|_{\mathbb{R}D + E \wedge E'}
                    - \frac{1}{2} \sum\limits_{i=1}^{p} ( \omega e_{i} s ) \otimes e_{i}^{*} + \frac{1}{2} \sum\limits_{j=1}^{q} ( \omega e'_{j} s ) \otimes {e'_{j}}^{*}$
                    for any $s \in W$,
        \end{enumerate}
        where the isomorphism $\mathfrak{m} \simeq \mathfrak{m}^{*}$ is induced by the scalar product $g$.
    \end{Theorem}
    \noindent\textit{Proof:} 
    To compute $g(L_{x} y, z)$ we will use the formula (\ref{NomizuOperator}). Due to linearity it is enough to consider only
    the basis elements of $\mathfrak{m}$. Moreover, the form of the formula for $g(L_{x} y, z)$ essentially depends only
    on what irreducible $\text{Ad}_{K}$-submodules of $\mathfrak{m}$ the elements $y, z$ belong to. Further when we say that the pair of
    irreducible submodules $(F_{1}, F_{2})$, $F_{1}, F_{2} \subset \mathfrak{m}$ is considering it means that we consider
    the case $y \in F_{1}$, $z \in F_{2}$. Also note that if for computation of some Nomizu operator $L_{x}$ we don't consider
    some pair $(F_{1}, F_{2})$ then it means that the corresponding values $g(L_{x} y, z)$ vanish. 

    To compute $L_{D}$ we need to consider only the pairs $(E, E)$, $(E', E')$, $(W, W)$. Say, for the pair $(E, E)$ the following
    quality holds 
    \begin{equation*}
        g(L_{D} e_{i}, e_{k}) = - \frac{1}{2} \big( g(e_{i},e_{k}) - g(e_{k}, e_{i}) \big) = 0.
    \end{equation*}
    For the pair $(E', E')$ and $(W, W)$ the computations are similar, so we obtain $L_{D} = 0$.

    Next, consider the basis element $x = e_{i} \wedge e'_{j} \in E \wedge E'$. For this case we need to consider the pairs
    $(E, E')$, $(E', E)$, $(W, W)$. For $(E, E')$ we have
    \begin{equation*}
        g(L_{e_{i} \wedge e'_{j}} e_{k}, e'_{s}) = -\frac{1}{2} \Big( g(-\delta_{ki}e'_{j},e'_{s}) + g(\delta_{js}e_{i}, e_{k}) \Big) = 0.
    \end{equation*}
    For $(E', E)$ similar computation gives zero as well. For $(W, W)$ we have
    \begin{equation*} 
        g(L_{e_{i} \wedge e'_{j}} s, s') = -\frac{1}{2} \Big( g([e_{i} \wedge e'_{j}, s], s')-g(s, [e_{i} \wedge e'_{j}, s']) \Big) = 0.
    \end{equation*}
    Therefore, we obtain the statement (i).

    To show (ii) we take the basis bivector $x = e_{i} \wedge e_{k} \in E \wedge E$. For this case we need to consider 
    the pairs $(E, E)$, $(E \wedge E', E \wedge E')$, $(W, W)$. For $(E, E)$ we have
    \begin{equation*}
        g(L_{e_{k} \wedge e_{l}} e_{i}, e_{r}) = -\frac{1}{2} \Big( g(\delta_{li}e_{k} - \delta_{ki}e_{l}, e_{r}) + g(-\delta_{lr}e_{k} + \delta_{kr}e_{l}, e_{i}) \Big)
            = -\delta_{li}\delta_{kr} + \delta_{ki}\delta_{lr},
    \end{equation*}
    so $L_{e_{k} \wedge e_{l}} e_{i} = - \delta_{li} e_{k} + \delta_{ki} e_{l} = -\text{ad}_{e_{k}\wedge e_{l}}e_{i}$.
    For $(E \wedge E', E \wedge E')$ we have
    \begin{equation*}
        g(L_{e_{k} \wedge e_{l}} e_{i} \wedge e'_{j}, e_{r} \wedge e'_{s}) =
            - \frac{1}{2} \big( g(\delta_{li} e_{k} \wedge e'_{j} - \delta_{ik} e_{l} \wedge e'_{j}, e_{r} \wedge e'_{s})
            + g(\delta_{rk} e_{l} \wedge e'_{s} - \delta_{lr} e_{k} \wedge e'_{s}, e_{i} \wedge e'_{j}) \big)
    \end{equation*}
    \begin{equation*}
            = - \frac{1}{2} ( \delta_{li}\delta_{kr}\delta_{js} - \delta_{ik}\delta_{lr}\delta_{js}
                + \delta_{rk}\delta_{li}\delta_{js} - \delta_{lr}\delta_{ki}\delta_{js} ) 
            = \big( \delta_{ki}\delta_{lr} - \delta_{kr}\delta_{li} \big) \delta_{js}, 
    \end{equation*}
    which leads to $L_{e_{k} \wedge e_{l}} e_{i} \wedge e'_{j} = - (\delta_{li} e_{kj} - \delta_{ki} e_{lj}) = - \text{ad}_{e_{k} \wedge e_{l}} e_{i} \wedge e'_{j}$.
    Finally, for $(W, W)$ we have
    \begin{equation*}
        g(L_{e_{k} \wedge e_{l}} s, s') = -\frac{1}{2} \Big( g([e_{k}\wedge e_{l},s],s') - g(s,[e_{k}\wedge e_{l},s']) \Big)
            = -g([e_{k} \wedge e_{l}, s], s'),
    \end{equation*}
    so $L_{e_{k} \wedge e_{l}} s = - \text{ad}_{e_{k} \wedge e_{l}} s$. Therefore, we obtain that $L_{e_{k} \wedge e_{l}} = - \text{ad}_{e_{k} \wedge e_{l}}$.
    For the basis bivectors of the space $E' \wedge E'$ all computations are similar, so we obtain the statement (ii).

    Now take the basis vector $x = e_{i} \in E$. Non-zero computations are given by only the following pairs:
    $(\mathbb{R}D, E)$, $(E, \mathbb{R}D)$, $(E \wedge E', E')$, $(E', E \wedge E')$, $(W, W)$.
    Below is the corresponding computations
    \begin{equation*}
        g(L_{e_{i}} D, e_{k}) = -\frac{1}{2} \Big( g([e_{i}, D], e_{k}) - g([D, e_{k}], e_{i}) \Big) = \delta_{ik},
    \end{equation*}
    \begin{equation*}
        g(L_{e_{i}} e_{k}, D) = - g(L_{e_{i}} D, e_{k})  = -\delta_{ik}, 
    \end{equation*}
    \begin{equation*}
        g(L_{e_{k}} e_{i} \wedge e'_{j}, e'_{s}) = -\frac{1}{2} \Big( g(\delta_{ki}e'_{j}, e'_{s}) - g(-\delta_{js}e_{i}, e_{k})\Big)
            = -\frac{1}{2} (\delta_{ki}\delta_{js}+\delta_{js}\delta_{ik})= -\delta_{ki}\delta_{js},
    \end{equation*}
    \begin{equation*}
        g(L_{e_{k}} e'_{s}, e_{i} \wedge e'_{j}) = - g(L_{e_{k}} e_{i} \wedge e'_{j}, e'_{s}) = \delta_{ki} \delta_{js},
    \end{equation*}
    \begin{equation*}
        g(L_{e_{i}}s, s') = \frac{1}{2} \Big( \langle [s, s'], e_{i} \rangle \Big) = \frac{1}{2} \Big( b(\omega e_{i} s, s') \Big) 
            = g \Big( - \frac{1}{2}\omega e_{i} s, s' \Big).
    \end{equation*}
    These imply the statement (iii), but (iv) can be proved by analogous computations.

    To show (v) we need to consider the following pairs: $(\mathbb{R}D, W)$, $(W, \mathbb{R}D)$,
    $(W, E)$, $(E, W)$, $(W, E')$, $(E', W)$, $(E \wedge E', W)$, $(W, E \wedge E')$. Also note that
    $\omega^{2} = (-1)^{\frac{p(p+1)}{2}} = 1$ for $p \equiv 3 \pmod{4}$. All necessary computations are given below:
    \begin{equation*}
        g(L_{s}D, s') = -\frac{1}{2} \Big( g([s, D], s') - g([D, s'], s) \Big) = \frac{1}{2}g(s,s'),
    \end{equation*}
    \begin{equation*}
        g(L_{s}s', D) = - g(L_{s}D, s') = - \frac{1}{2} g(s, s'),
    \end{equation*}
    \begin{equation*}
        g(L_{s} s', e_{i}) = -\frac{1}{2} \Big( g([s, s'], e_{i}) \Big) = g\Big( -\frac{1}{2} [s, s'], e_{i} \Big),
    \end{equation*}
    \begin{equation*}
        g(L_{s} e_{i}, s') = \frac{1}{2}g([s, s'], e_{i}) = \frac{1}{2} \langle [s, s'], e_{i} \rangle 
            = \frac{1}{2}b(\omega e_{i} s, s') = -\frac{1}{2}g(\omega e_{i} s, s'),
    \end{equation*}
    \begin{equation*}
        g(L_{s} s', e'_{j}) = -\frac{1}{2} \Big( g([s, s'], e'_{j}) \Big) = g\Big( -\frac{1}{2} [s, s'] , e'_{j} \Big),
    \end{equation*}
    \begin{equation*}
        g(L_{s}e'_{j}, s') = \frac{1}{2}g([s, s'], e'_{j}) = -\frac{1}{2} \langle [s, s'], e'_{j} \rangle 
            = -\frac{1}{2}b(\omega e'_{j} s, s') = \frac{1}{2}g(\omega e'_{j} s, s'),
    \end{equation*}
    \begin{equation*}
        g(L_{s} ( e_{i} \wedge e'_{j}), s') = - \frac{1}{4} \Big( g(e_{i}e'_{j}s, s') + g(s, e_{i}e'_{j} s') \Big)
            = \frac{1}{4} \Big( b(e_{i}e'_{j}s, s') + b(s, e_{i}e'_{j} s') \Big)
    \end{equation*}
    \begin{equation*}
            = \frac{1}{4} \Big( b(e_{i}e'_{j}s, s') + b(e_{i}e'_{j} s, s') \Big) = \frac{1}{2} b(e_{i}e'_{j}s, s') = - g([s, e_{i} \wedge e'_{j}], s'),
    \end{equation*}
    \begin{equation*}
        g(L_{s}s', e_{i} \wedge e'_{j}) = -\frac{1}{2} \Big( -g([s', e_{i} \wedge e'_{j}], s) + g([e_{i} \wedge e'_{j}, s], s') \Big)
    \end{equation*}
    \begin{equation*}
            = \frac{1}{4} \Big( g(e_{i}e'_{j}s, s') + g(s, e_{i}e'_{j}s') \Big) = \frac{1}{2} g(e_{i}e'_{j}s, s').
    \end{equation*}
    \hfill$\Box$

    \begin{Remark}
        Theorem \ref{NomizuOperatorsTheorem} is consistent with \cite[Lemma 10]{C} for $p = 3$. Note that for $p > 3$
        peuso-Riemannian metric is considered in \cite{C} which differs from what is considered in the present paper.
    \end{Remark}

    \subsection{Computation of Ricci curvature}
    Using explicit expressions of Nomizu operators we are able to compute the Ricci curvature of the metric $g_{M_{p,q}}$.
    This computation will show that the metric $g_{M_{p,q}}$ is Einstein.
    \begin{Theorem}\label{MainTheorem}
        Noncompact Riemannian homogeneous space $(M_{p,q}, g_{M_{p,q}})$ is an Einstein manifold with the constant
        $\lambda = - p - q - N(p + q - 3)$.
    \end{Theorem}
    \noindent\textit{Proof:} From the formula (\ref{RiemannTensorFormula}) one can derive the following expression
    for the Ricci curvature at the point $o \in M_{p,q}$ using Nomizu operators
    \begin{equation}\label{RicciQuadraticForm}
        \text{Ric}(x, x) = \text{tr}\big( - R_{x}^{2} + R_{L_{x}x} - R_{x}\text{ad}_{x}^{\mathfrak{k}} \big)
    \end{equation}
    where $\text{ad}_{x}^{\mathfrak{k}} := \pi^{\mathfrak{k}} \circ \text{ad}_{x}$, $\pi^{\mathfrak{k}}$ is the projection operator
    from $\mathfrak{g}(W, \Pi)$ onto $\mathfrak{k}$ along $\mathfrak{m}$, i. e. $\pi^{\mathfrak{k}} = \text{Id} - \pi$.
    In the formula (\ref{RicciQuadraticForm}) and in all formulas below we consider the trace which is taken over the subspace
    $\mathfrak{m} \subset \mathfrak{g}(W, \Pi)$.

    Firstly, we compute the term $L_{x}x$. According to the formula (\ref{NomizuOperator}) we have the equality 
    \begin{equation*}
        g(L_{x}x, y) = g(\pi[x,y], x).
    \end{equation*}
    The right handside of this equality is not equal to zero only in the following cases: 
    $(x, y) \in (E, \mathbb{R}D)$, $(E', \mathbb{R}D)$, $(W, \mathbb{R}D)$, $(W, E \wedge E')$.
    For simplicity we will always consider $s \in W$ with $g(s, s) = 1$ unless otherwise stated.
    Consider below each case separately
    \begin{equation*}
        g(L_{e_{i}}e_{i}, D) = g(\pi[e_{i}, D], e_{i}) = -1,
    \end{equation*}
    \begin{equation*}
        g(L_{e'_{j}}e'_{j}, D) = g(\pi[e'_{j}, D], e'_{j}) = -1,
    \end{equation*}
    \begin{equation*}
        g(L_{s}s, D) = g(\pi[s, D], s) = -\frac{1}{2},
    \end{equation*}
    \begin{equation*}
        g(L_{s}s, e_{i} \wedge e'_{j}) = g(\pi[s, e_{i} \wedge e'_{j}], s) = \frac{1}{2} g(e_{i}e'_{j}s, s).
    \end{equation*}
    From this we deduce the formulas
    \begin{equation}
        L_{e_{i}}e_{i} = L_{e'_{j}}e'_{j} = - D, \ \ \ 
        L_{s}s = -\frac{1}{2} D + \frac{1}{2} \sum\limits_{i,j=1}^{p,q} g(e_{i}e'_{j}s, s) e_{i} \wedge e'_{j}. 
    \end{equation}
    Using the observation that $R_{x} = L_{x} \circ \pi + \pi \circ \text{ad}_{x}$ and the previous results $L_{D} = 0, L_{e_{i} \wedge e'_{j}} = 0$
    we have that
    \begin{equation}
        R_{D} = \pi \circ \text{ad}_{D} = \text{ad}_{D}, \ \ \ R_{e_{i} \wedge e'_{j}} = \pi \circ \text{ad}_{e_{i} \wedge e'_{j}},
    \end{equation}
    which implies the formulas
    \begin{equation}
        R_{L_{e_{i}}e_{i}} = R_{L_{e'_{j}}e'_{j}} = - R_{D} = - \text{ad}_{D},
    \end{equation}
    \begin{equation*}
        R_{L_{s}s} = -\frac{1}{2} R_{D} + \sum\limits_{i,j=1}^{p,q} g(e_{i}e'_{j}s, s)R_{e_{i}\wedge e'_{j}}  
            = - \frac{1}{2} \text{ad}_{D} + \frac{1}{2} \sum\limits_{i,j=1}^{p,q} g(e_{i}e'_{j}s, s) \pi \circ \text{ad}_{e_{i}\wedge e'_{j}}. 
    \end{equation*}
    Now the trace of the operators $R_{L_{e_{i}}e_{i}}$ and $R_{L_{e'_{j}}e'_{j}}$ can be easily computed
    \begin{equation}\label{TrRLEE}
        \text{tr } R_{L_{e_{i}}e_{i}} = \text{tr } R_{L_{e'_{j}}e'_{j}} = - n - \frac{d}{2}.
    \end{equation}
    To compute the trace of the operator $R_{L_{s}s}$ note that $\text{ad}_{e_{i} \wedge e'_{j}}$ is skew-symmetric endomorphism
    of $V = E \oplus E'$ with respect to the original scalar product $\langle -, - \rangle$, so $\text{tr }\text{ad}_{e_{i} \wedge e'_{j}}|_{V} = 0$.
    To see that $\text{tr } \text{ad}_{e_{i} \wedge e'_{j}}|_{W} = 0$ holds as well we consider the auxilliary bilinear form $h_{kl}$
    defined for any indicies $1 \leq k \leq p$, $1 \leq l \leq q$, as follows
    \begin{equation}\label{AuxilliaryBilinearForm}
        h_{kl}(s, s') := b(e_{k} e'_{l} s, s').
    \end{equation}
    Due to the properties (ii) and (iii) from Theorem \ref{b_properties} this bilinear form is symmetric and nondegenerate.
    Let $k$ be any index from $\{1, .., p\}$ different from $i$, i. e. $k \in \{1, .., p\} \setminus \{ i \}$, and $l$ is equal to $j$.
    For this case we have the following equality
    \begin{equation*}
        h_{kj}(\text{ad}_{e_{i} \wedge e'_{j}} s, s) = \frac{1}{2} b(e_{k}e_{i} s, s) = - \frac{1}{2} b(s, e_{k}e_{i} s)
    \end{equation*}
    \begin{equation*}
        = - \frac{1}{2} b(e_{k}e_{i} s, s) = -h_{kj}(\text{ad}_{e_{i} \wedge e'_{j}} s, s),
    \end{equation*}
    which implies $h_{kj}(\text{ad}_{e_{i} \wedge e'_{j}} s, s) = 0$ for any $s \in W$. Consequently, the computation of the trace
    of $\text{ad}_{e_{i} \wedge e'_{j}}$ in any $h_{kj}$-orthonormal basis vanishes, so we indeed have $\text{tr } \text{ad}_{e_{i} \wedge e'_{j}}|_{W} = 0$.
    So we have the equality 
    \begin{equation}
        \text{tr } R_{L_{s}s} = - \frac{n}{2} - \frac{d}{4}. 
    \end{equation}

    Now consider the term $R_{x}\text{ad}_{x}^{\mathfrak{k}}$. It is not equal to zero if and only if $x \in E \wedge E'$
    because $\text{ad}_{x}^{\mathfrak{k}}$ vanishes on other irreducible subspaces of $\mathfrak{m}$.
    More precisely, $\text{ad}_{x}^{\mathfrak{k}} y$ vanishes for any $x, y \in (E \wedge E')^{\perp}$.
    On the other hand, for $x = e_{i} \wedge e'_{j}$ we have
    \begin{equation*}
        R_{e_{i} \wedge e'_{j}}\text{ad}_{e_{i} \wedge e'_{j}}^{\mathfrak{k}} e_{k} \wedge e'_{l} 
            = \pi \circ \text{ad}_{e_{i} \wedge e'_{j}} (- \delta_{ik} e'_{j} \wedge e'_{l} +\delta_{jl}e_{i}\wedge e_{k})
   \end{equation*}
   \begin{equation*}
            = \delta_{ik} [e'_{j} \wedge e'_{l}, e_{i} \wedge e'_{j}] - \delta_{jl} [e_{i} \wedge e_{k}, e_{i} \wedge e'_{j}]
    \end{equation*}
    \begin{equation*}
            = \delta_{ik} e_{i} \wedge e'_{l} + \delta_{jl} e_{k} \wedge e'_{j} - 2 \delta_{ik} \delta_{jl} e_{i} \wedge e'_{j}.
    \end{equation*}
    From this we obtain that $g(R_{e_{i} \wedge e'_{j}}\text{ad}_{e_{i} \wedge e'_{j}}^{\mathfrak{k}}  e_{k} \wedge e'_{l}, e_{k} \wedge e'_{l})
    = \delta_{ik}+\delta_{jl} - 2 \delta_{ik}\delta_{jl}$ which implies
    \begin{equation}\label{MixedTrEE'}
        \text{tr } R_{e_{i} \wedge e'_{j}}\text{ad}_{e_{i} \wedge e'_{j}}^{\mathfrak{k}} = q + p - 2 = n-2. 
    \end{equation}
    Therefore, at the moment the following formulas hold
    \begin{equation}\label{Ric(D, D)}
        \text{Ric}(D,D) = -\text{tr }R_{D}^{2} = -\text{tr} (\text{ad}_{D})^{2} = -n - \frac{d}{4},
    \end{equation}
    \begin{equation}\label{Ric_E_E'}
        \text{Ric}(e_{i}, e_{i}) = -\text{tr } R_{e_{i}}^{2} - n - \frac{d}{2}, \ \ \
        \text{Ric}(e'_{j}, e'_{j}) = -\text{tr } R_{e'_{j}}^{2} - n - \frac{d}{2},
    \end{equation}
    \begin{equation}
        \text{Ric}(e_{i} \wedge e'_{j}, e_{i} \wedge e'_{j}) = - \text{tr } R_{e_{i} \wedge e'_{j}}^{2} - n + 2,
    \end{equation}
    \begin{equation}\label{Ric_W}
        \text{Ric}(s, s) = - \text{tr } R_{s}^{2} - \frac{n}{2} - \frac{d}{4}.
    \end{equation}

    From the equality $R_{e_{i} \wedge e'_{j}} = \pi \circ \text{ad}_{e_{i} \wedge e'_{j}}$ we see that the trace $\text{tr } R_{e_{i} \wedge e'_{j}}^{2}$
    can be computed like below
    \begin{equation}\label{TrREE'2}
        \text{tr } R_{e_{i} \wedge e'_{j}}^{2} = \sum\limits_{k=1}^{p} g\big( (\text{ad}_{e_{i} \wedge e'_{j}})^{2} e_{k}, e_{k} \big)
                 + \sum\limits_{l=1}^{q} g\big( (\text{ad}_{e_{i} \wedge e'_{j}})^{2} e'_{l},e'_{l} \big)
    \end{equation}
    \begin{equation*}
                + \sum\limits_{r=1}^{d} g((\text{ad}_{e_{i} \wedge e'_{j}})^{2} s_{r}, s_{r})
            = \sum\limits_{k=1}^{p} g\big( \delta_{ik} e_{i}, e_{k} \big) + \sum\limits_{l=1}^{q} g\big( \delta_{jl} e'_{j}, e'_{l} \big)
    \end{equation*}
    \begin{equation*}
                + \frac{1}{4}\sum\limits_{r=1}^{d} g((e_{i}e'_{j})^{2} s_{r}, s_{r})
            = \sum\limits_{k=1}^{p} \delta_{ik} + \sum\limits_{l=1}^{q} \delta_{jl} + \frac{1}{4}\sum\limits_{r=1}^{d} g(s_{r}, s_{r})
            =  2 + \frac{d}{4},
    \end{equation*}
    so we have 
    \begin{equation*}
        \text{Ric}(e_{i} \wedge e'_{j}, e_{i} \wedge e'_{j}) = -\text{tr} (\pi \circ \text{ad}_{e_{i} \wedge e'_{j}})^{2} - n + 2 = - n - \frac{d}{4}.
    \end{equation*}
    Next, we compute $\text{tr} R_{e_{i}}^{2}$ and $\text{tr} R_{e'_{j}}^{2}$. Firstly, we note that  
    \begin{equation*}
        R_{e_{i}} = L_{e_{i}} + \pi\circ\text{ad}_{e_{i}} = -D \otimes e_{i}^{*} + 
                \sum\limits_{j=1}^{q} e_{i} \wedge e'_{j} \otimes {e'_{j}}^{*} - \frac{1}{2}\omega e_{i},
    \end{equation*}
    \begin{equation*}
        R_{e'_{j}} = L_{e'_{j}} + \pi\circ\text{ad}_{e'_{j}} = -D \otimes {e'_{j}}^{*} 
            + \sum\limits_{i=1}^{p} e_{i} \wedge e'_{j} \otimes e_{i}^{*} + \frac{1}{2}\omega e'_{j},
    \end{equation*}
    From these formulas we see that
    \begin{equation}\label{BlockRE}
        R_{e_{i}}: \mathfrak{k} \rightarrow 0, \ \ E \rightarrow \mathbb{R}D, \ \ E'\rightarrow E\wedge E', \ \ W \rightarrow W,
    \end{equation}
    \begin{equation*}
        R_{e'_{j}}: \mathfrak{k} \rightarrow 0, \ \ E' \rightarrow \mathbb{R}D, \ \ E \rightarrow E \wedge E', \ \ W \rightarrow W.
    \end{equation*}
    So the operators $R_{e_{i}}^{2}$, $R_{e'_{j}}^{2}$ can be considered as endomorphisms of  $W$,
    i. e. $R_{e_{i}}^{2}, R_{e'_{j}}^{2}: W \rightarrow W$. Therefore, we have
    \begin{equation}\label{TrRe2}
        \text{tr } R_{e_{i}}^{2} = \text{tr } R_{e_{i}}^{2}|_{W} = \sum\limits_{u=1}^{d} g\big( R_{e_{i}}^{2} s_{u}, s_{u} \big)
            = \frac{1}{4} \sum\limits_{u=1}^{d} g\big( \omega e_{i} \omega e_{i} s_{u}, s_{u} \big) = - \frac{d}{4},
    \end{equation}
    The trace of $R_{e'_{j}}^{2}$ can be computated by the similar computation
    \begin{equation}\label{tr_R_e'}
        \text{tr } R_{e'_{j}}^{2} = \sum\limits_{u=1}^{d} g\big( R_{e'_{j}}^{2} s_{u}, s_{u} \big)
            = \frac{1}{4} \sum\limits_{u=1}^{d} g\big( \omega e'_{j} \omega e'_{j} s_{u}, s_{u} \big) = - \frac{d}{4}.
    \end{equation}
    Substituting these values to the expressions (\ref{Ric_E_E'}) we obtain
    \begin{equation}
        \text{Ric}(e_{i}, e_{i}) = \text{Ric}(e'_{j}, e'_{j}) = \frac{d}{4} - n - \frac{d}{2} = - n - \frac{d}{4}.
    \end{equation}

    Finally, we compute the summand $\text{tr } R_{s}^{2}$ in the formula (\ref{Ric_W}). We have 
    \begin{equation*}
        R_{s} = L_{s} + \pi \circ \text{ad}_{s} = -\frac{1}{2} D \otimes s^{*} 
                + \frac{1}{2} \sum\limits_{i=1}^{p}\sum\limits_{j=1}^{q} e_{i} \wedge e'_{j} \otimes (e_{i}e'_{j}s)^{*} + \frac{1}{2}\text{ad}_{s}|_{W}
    \end{equation*}
    \begin{equation*}
                - \frac{1}{2} \sum\limits_{i=1}^{p} \omega e_{i} s \otimes e_{i}^{*} 
                + \frac{1}{2} \sum\limits_{j=1}^{q} \omega e'_{j} s \otimes {e'_{j}}^{*}.
    \end{equation*}
    From this we see that
    \begin{equation*}
        R_{s}: \mathfrak{k} \rightarrow W, \ \ E \rightarrow W, \ \ E' \rightarrow W, \ \ W \rightarrow \mathbb{R}D+E+E'+E \wedge E'
    \end{equation*}
    which implies that only subspaces $E$, $E'$, $W$ can give a non-zero contribution to the computation of $\text{tr }R_{s}^{2}$.
    More precisely, for any $s \in W$ with $g(s, s) = 1$ we have
    \begin{equation*}
        \text{tr } R_{s}^{2}|_{E} = \sum\limits_{i=1}^{p} g(R_{s}^{2} e_{i}, e_{i}) 
            = - \frac{1}{4} \sum\limits_{i=1}^{p} \langle [s, \omega e_{i} s], e_{i} \rangle
    \end{equation*}
    \begin{equation*}
            = - \frac{1}{4} \sum\limits_{i=1}^{p} b(\omega e_{i} s, \omega e_{i} s) = - \frac{1}{4} \sum\limits_{i=1}^{p} b(s, s)
            = \frac{1}{4} \sum\limits_{i=1}^{p} g( s, s) = \frac{p}{4},
    \end{equation*}
    \begin{equation*}
        \text{tr } R_{s}^{2}|_{E'} = \sum\limits_{j=1}^{q} g(R_{s}^{2} e'_{j}, e'_{j})
            = - \frac{1}{4} \sum\limits_{j=1}^{q} \langle [s, \omega e'_{j} s], e'_{j} \rangle
    \end{equation*}
    \begin{equation*}
            = - \frac{1}{4} \sum\limits_{j=1}^{q} b( \omega e'_{j} s, \omega e'_{j} s) = - \frac{1}{4} \sum\limits_{j=1}^{q} b(s, s)
            = \frac{1}{4} \sum\limits_{j=1}^{q} g( s, s ) = \frac{q}{4},
    \end{equation*}
    \begin{equation*}
        \text{tr } R_{s}^{2}|_{W} = \sum\limits_{k = 1}^{d} g(R_{s}^{2} s_{k}, s_{k}) = \frac{1}{2} \sum\limits_{k = 1}^{d} g(R_{s}[s, s_{k}], s_{k})
    \end{equation*}
    \begin{equation*}
            = - \frac{1}{4} \sum\limits_{k=1}^{d} \sum\limits_{i=1}^{p} g\Big(\omega e_{i} s \langle [s, s_{k}], e_{i} \rangle, s_{k} \Big)
                - \frac{1}{4} \sum\limits_{k=1}^{d} \sum\limits_{j=1}^{q} g\Big(\omega e'_{j} s \langle [s, s_{k}], e'_{j} \rangle, s_{k} \Big)
    \end{equation*}
    \begin{equation*}
            = - \frac{1}{4} \sum\limits_{k=1}^{d} \sum\limits_{i=1}^{p} g\Big(\omega e_{i} s b(\omega e_{i} s, s_{k}), s_{k} \Big)
                - \frac{1}{4} \sum\limits_{k=1}^{d} \sum\limits_{j=1}^{q} g\Big(\omega e'_{j} s b( \omega e'_{j} s, s_{k}), s_{k} \Big)
    \end{equation*}
    \begin{equation*}
            = \frac{1}{4} \sum\limits_{k=1}^{d} \sum\limits_{i=1}^{p} g\Big(\omega e_{i} s, g(\omega e_{i} s, s_{k}) s_{k} \Big)
                + \frac{1}{4} \sum\limits_{k=1}^{d} \sum\limits_{j=1}^{q} g\Big(\omega e'_{j} s, g( \omega e'_{j} s, s_{k}) s_{k} \Big)
    \end{equation*}
    \begin{equation*}
            = \frac{1}{4} \sum\limits_{i=1}^{p} g\Big(\omega e_{i} s, \omega e_{i} s \Big)
                + \frac{1}{4} \sum\limits_{j=1}^{q} g\Big(\omega e'_{j} s, \omega e'_{j} s \Big)
            = - \frac{1}{4} \sum\limits_{i=1}^{p} b\Big(\omega e_{i} s, \omega e_{i} s \Big)
                - \frac{1}{4} \sum\limits_{j=1}^{q} b\Big(\omega e'_{j} s, \omega e'_{j} s \Big)
    \end{equation*}
    \begin{equation*}
            = - \frac{1}{4} \sum\limits_{i=1}^{p} b(s, s) - \frac{1}{4} \sum\limits_{j=1}^{q} b( s, s )
            = \frac{1}{4} \sum\limits_{i=1}^{p} g ( s, s ) + \frac{1}{4} \sum\limits_{j=1}^{q} g( s, s ) = \frac{n}{4},
    \end{equation*}
    \begin{equation*}
        \text{tr }R_{s}^{2} = \text{tr } R_{s}^{2}|_{E} + \text{tr } R_{s}^{2}|_{E'} + \text{tr } R_{s}^{2}|_{W}
            = \frac{p}{4} + \frac{q}{4} + \frac{n}{4} = \frac{n}{2}.
    \end{equation*}
    
    Therefore, we obtain that
    \begin{equation}
        \text{Ric}(s, s) = - \text{tr } R_{s}^{2} - \frac{n}{2} - \frac{d}{4} = - \frac{n}{2} - \frac{n}{2} - \frac{d}{4} = - n - \frac{d}{4}.
    \end{equation}

    Finally, we have the equalities
    \begin{equation}\label{Ric_equality}
        \text{Ric}(D, D) = \text{Ric}(e_{i}, e_{i}) = \text{Ric}(e'_{j}, e'_{j})
    \end{equation}
    \begin{equation*}
            = \text{Ric}(e_{i} \wedge e'_{j}, e_{i} \wedge e'_{j}) = \text{Ric}(s, s) = - n - \frac{d}{4}.
    \end{equation*}
    Since the last equality in this sequence is checked for any $s \in W$ we obtain that $\lambda \cdot g$ and $\text{Ric}$
    coincide over $W$. On the other hand, these bilinear forms are $\text{Ad}_{K}$-invariant and the $\text{Ad}_{K}$-modules
    $\mathbb{R}D$, $E$, $E'$, $E \wedge E' \subset \mathfrak{m}$ are irreducible, so we obtain that $\lambda \cdot g$
    and $\text{Ric}$ coincide over each of these irreducible submodules. Moreover, if $q > 1$ then all of them are pairwise
    non-equivalent which necesserily implies that all of them are orthogonal with respect to both $g$ and $\text{Ric}$.
    Therefore, for $q > 1$ we obtain that $\lambda \cdot g$ and $\text{Ric}$ coincide over the whole tangent space $\mathfrak{m}$.

    Additional check is only needed for the case $q = 1$. Namely, we need to show that in this case we have $\text{Ric}(D, e'_{1}) = 0$ and
    $\text{Ric}(e_{i}, e_{j} \wedge e'_{1}) = 0$ for any $1 \leq i, j \leq p$. Using (\ref{RicciQuadraticForm}), (\ref{Ric_equality}),
    (\ref{Ric(D, D)}), (\ref{tr_R_e'}) and taking into account the equalities $\text{ad}_{D+e'_{1}}^{\mathfrak{k}} = 0$
    and $L_{D+e'_{1}}(D+e'_{1}) = L_{e'_{1}}e'_{1} = -D$ we obtain
    \begin{equation*}
        \text{Ric}(D+e'_{1}, D+e'_{1}) = - \text{tr}\bigg( R_{D}^{2} + 2R_{D}R_{e'_{1}} + R_{e'_{1}}^{2} + R_{D}  \bigg)
    \end{equation*}
    \begin{equation*}
        = - \text{tr} \big( \text{ad}_{D}R_{e'_{1}} \big) - 2 \bigg( n + \frac{d}{4} \bigg).
    \end{equation*}
    We have $\text{tr} \big( \text{ad}_{D}R_{e'_{1}} \big) = \frac{1}{4}\sum\limits_{u=1}^{N} g(\omega e'_{1} s_{u}, s_{u}) = 0$
    because $\omega e'_{1}$ acts on $W$ by $g$-skew-symmetric endomorphism due to Theorem \ref{b_properties}.
    Therefore, we obtain $\text{Ric}(D+e'_{1}, D+e'_{1}) = \text{Ric}(D, D) + \text{Ric}(e'_{1}, e'_{1})$ which implies $\text{Ric}(D, e'_{1}) = 0$.

    Similarly, using (\ref{TrRe2}), (\ref{TrREE'2}), (\ref{TrRLEE}), (\ref{BlockRE}), (\ref{MixedTrEE'}) we obtain
    \begin{equation*}
        \text{Ric}(e_{i} + e_{j} \wedge e'_{1}, e_{i} + e_{j} \wedge e'_{1}) =
    \end{equation*}
    \begin{equation*}
            = \text{tr}\big( -R_{e_{i}}^{2} - 2 R_{e_{i}} R_{e_{j} \wedge e'_{1}} - R_{e_{j} \wedge e'_{1}}^{2}
                + R_{L_{e_{i}} (e_{i} + e_{j} \wedge e'_{1})} - R_{e_{i}+e_{j} \wedge e'_{1}} \text{ad}^{\mathfrak{k}}_{e_{j} \wedge e'_{1}}  \big)
    \end{equation*}
    \begin{equation*}
            = -2 \bigg( n + \frac{d}{4} \bigg) - 2 \text{tr}\big( R_{e_{i}} R_{e_{j} \wedge e'_{1}} \big) + \text{tr} \big( R_{L_{e_{i}} e_{j} \wedge e'_{1}} \big)
    \end{equation*}
    From the statement (iii) of Theorem \ref{NomizuOperatorsTheorem} we have $L_{e_{i}} e_{j} \wedge e'_{1} = - \delta_{ij} e'_{1}$,
    so $\text{tr} \big( R_{L_{e_{i}} e_{j} \wedge e'_{1}} \big) = - \delta_{ij} \text{tr } R_{e'_{1}} = 0$. Next, note that for the bilinear form
    (\ref{AuxilliaryBilinearForm}) with $k = i$ and $l = 1$ we have the following equality 
    \begin{equation*}
             h_{i1}(R_{e_{i}} R_{e_{j} \wedge e'_{1}} s, s) = \frac{1}{4} b(e_{i}e'_{1}\omega e_{i}e_{j}e'_{1} s, s)
                 = \frac{1}{4} b(\omega e_{j} s, s)
    \end{equation*}
    \begin{equation*}
                = - \frac{1}{4} b( s, \omega e_{j} s) = - \frac{1}{4} b(\omega e_{j} s, s) = - h_{i1}(R_{e_{i}} R_{e_{j} \wedge e'_{1}} s, s)
    \end{equation*}
    which means that $h_{i1}(R_{e_{i}} R_{e_{j} \wedge e'_{1}} s, s) = 0$ for any $s \in W$. Since $h_{i1}$ is symmetric and
    non-degenerate we obtain that $\text{tr}\big( R_{e_{i}} R_{e_{j} \wedge e'_{1}} \big) = 0$, so we have
    \begin{equation*} 
        \text{Ric}(e_{i} + e_{j} \wedge e'_{1}, e_{i} + e_{j} \wedge e'_{1}) = \text{Ric}(e_{i}, e_{i}) + \text{Ric}(e_{j} \wedge e'_{1}, e_{j} \wedge e'_{1}),
    \end{equation*}
    in particular, $\text{Ric}(e_{i}, e_{j} \wedge e'_{1}) = 0$ for any $1 \leq i, j \leq p$. Consequently, we showed that for
    the case $q = 1$ the bilinear forms $\lambda \cdot g$ and $\text{Ric}$ coincide over the whole tangent space $\mathfrak{m}$.
    Therefore, we proved that the Riemannian metric $g_{M_{p, q}}$ on $M_{p, q}$ is indeed Einstein with the Einstein constant
    $\lambda = - n - \frac{d}{4}$. Using the function $N(k)$ defined in Theorem \ref{DimensionRank} this constant can be
    expressed as $\lambda = -p - q - N(p+q-3)$.
    \hfill$\Box$


\begin{thebibliography}{99}

\bibitem{A}
    D. V. Alekseevskii,
    Classification of quaternionic spaces with a transitive solvable group of motions,
    Math. USSR-Izv., 9:2 (1975), 297–339.

\bibitem{A1}
    D. V. Alekseevskii,
    Homogeneous Riemannian spaces of negative curvature,
    Math. USSR-Sb., 25:1 (1975), 87–109.

\bibitem{A2}
    D. V. Alekseevskii,
    Conjugacy of polar factorizations of Lie groups,
    Math. USSR-Sb., 13:1 (1971), 12–24.

\bibitem{AK}
    D. V. Alekseevskii, B. N. Kimel'fel'd,
    Structure of homogeneous Riemann spaces with zero Ricci curvature,
    Funct. Anal. Appl., 9:2 (1975), 97–102.

\bibitem{AVL}
    D. V. Alekseevskii, A. M. Vinogradov, V. V. Lychagin,
    Basic ideas and concepts of differential geometry,
    Geometry – 1, Itogi Nauki i Tekhniki. Ser. Sovrem. Probl. Mat. Fund. Napr., 28, VINITI, Moscow, 1988, 5–289.

\bibitem{G}
    V. V. Gorbatsevich,
    On Some Classes of Bases in Finite-Dimensional Lie Algebras,
    Math. Notes, 114:2 (2023), 165–171.

\bibitem{AC}
    D. V. Alekseevsky, V. C\'{o}rtes, 
    Classification of $N$-(Super)-Extended Poincar\'{e} Algebras and Bilinear Invariants of 
    the Spinor Representation of $\text{Spin}(p, q)$,
    Commun. Math. Phys. 183, 477-510 (1997). 

\bibitem{ADF}
  D. Alekseevsky, I. Dotti, C. Ferraris,
  Homogeneous Ricci positive 5-manifolds,
  Pacific J. Math. 175 (1996), no. 1, 1–12.

\bibitem{Ber}
  M. Berger, A Panoramic View of Riemannian Geometry, Springer-Verlag, New York, 2003.
  
\bibitem{BTV}
  J. Berndt, F. Tricerri, L. Vanhecke,
  Generalized Heisenberg Groups and Damek-Ricci Harmonic Spaces,
  Springer-Verlag, Berlin Heidelberg, 1995.

\bibitem{Bes}
  A. L. Besse, 
  Einstein manifolds, 
  Ergebnisse der Mathematik und ihrer Grenzgebiete (3), 10, Springer-Verlag, Berlin, 1987.

\bibitem{B}
  Ch. B\"{o}hm,
  Homogeneous Einstein metrics and simplicial complexes,
  J. Differential Geom. 67 (2004), no. 1, 79–165.

\bibitem{BK}
  Ch. B\"{o}hm, M. M. Kerr,
  Low-dimensional homogeneous Einstein manifolds,
  Trans. Amer. Math. Soc. 358 (2006), no. 4, 1455–1468.

\bibitem{BL}
  Ch. B\"{o}hm, R. A. Lafuente, 
  Non-compact Einstein manifolds with symmetry,
  Journal of the American Mathematical Society, 36 (3), 591-651, 2023.
  
\bibitem{BWZ}
  Ch. B\"{o}hm, M. Wang, W. Ziller,
  A variational approach for compact homogeneous Einstein manifolds,
  Geom. Funct. Anal. 14 (2004), no. 4, 681–733.

\bibitem{CR}
  D. Conti, F. A. Rossi,
  Construction of nice nilpotent Lie groups,
  J. Algebra, 525 (2019), 311–340.

\bibitem{C}
  V. C\'{o}rtes,
  A New Construction of Homogeneous Quaternionic Manifolds and Related Geometric Structures
  American Mathematical Society: Memoirs of the American Mathematical Society, 2000.

\bibitem{DR}
  E. Damek, F. Ricci,
  A class of nonsymmetric harmonic Riemannian spaces,
  Bull. Amer. Math. Soc. 27 (1992), 139-142.

\bibitem{Gr}
  M. M. Graev,
  The existence of invariant Einstein metrics on a compact homogeneous space,
  Trans. Moscow Math. Soc. 2012, 1–28.
      
\bibitem{Heb}
  J. Heber, 
  Noncompact homogeneous Einstein spaces, 
  Invent. Math. 133 (1998), no. 2, 279–352.
  
\bibitem{J}
  M. Jablonski,
  Homogeneous Einstein manifolds,
  Revista de la Uni\'{o}n Matem\'{a}tica Argentina,
  Vol. 64, No. 2, 2023, pp. 461–485.

\bibitem{K}
  A. Kaplan,
  Riemannian nilmanifolds attached to Clifford modules,
  Geom Dedicata 11, 127–136 (1981).

\bibitem{L1}
  J. Lauret,
  Einstein solvmanifolds are standard,
  Ann. of Math. (2) 172 (2010), no. 3, 1859–1877.

\bibitem{L2}
  J. Lauret, 
  Ricci soliton homogeneous nilmanifolds, 
  Math. Ann. 319 (2001), no. 4, 715–733.

\bibitem{L3}
  J. Lauret,
  Finding Einstein solvmanifolds by a variational method,
  Math. Z., 241 (2002), 83 – 99.

\bibitem{LW}
  J. Lauret, C. Will,
  Einstein solvmanifolds: existence and non-existence questions,
  Math. Ann. 350, 199–225 (2011).

\bibitem{LM}
  H. B. Lawson, M.-L. Michelson,
  Spin geometry,
  Princeton Univ. Press, 1989.

\bibitem{M}
  K. Mori,
  Einstein metrics on Boggino-Damek-Ricci type solvable Lie groups,
  Osaka J. Math. 39(2): 345-362 (June 2002).

\bibitem{Nik1}
  Y. Nikolayevsky,
  Einstein solvmanifolds and the pre-Einstein derivation,
  Trans. Amer. Math. Soc. 363 (2011), 3935-3958.
  
\bibitem{Nik2}
  Y. Nikolayevsky,
  Einstein solvmanifolds with a simple Einstein derivation,
  Geom. Dedicata, 135 (2008), 87 - 102.

\bibitem{Nik3}
  Y. Nikolayevsky,
  Einstein solmanifolds with free nilradical,
  Ann. Global Anal. Geom., 33 (2008), 71 – 87.
  
\bibitem{N1}
  Yu. G. Nikonorov,
  Compact homogeneous Einstein 7-manifolds,
  Geom. Dedicata 109 (2004), 7–30.

\bibitem{N2}
  Yu. G. Nikonorov,  E. D. Rodionov,
  Compact homogeneous Einstein 6-manifolds,
  Differential Geom. Appl. 19 (2003), no. 3, 369–378.  

\bibitem{P}
  T. Payne,
  The existence of soliton metrics for nilpotent Lie groups,
  Geom. Dedicata, 145 (2010), 71–88.
  
\bibitem{W1}
  C. Will,
  Rank-one Einstein solvmanifolds of dimension 7,
  Differential Geom. Appl., 19 (2003), 307 – 318.

\bibitem{W2}
  C. Will,
  The space of solvsolitons in low dimensions,
  Ann Glob Anal Geom 40, 291–309 (2011).  

\end{thebibliography}
\end{document}